\documentclass{amsart}
\usepackage{amsmath}
\usepackage{amssymb}
\usepackage{amscd}
\usepackage[all]{xy}
\usepackage{color}
\usepackage{tikz-cd}
\usepackage{ulem}
\usepackage{tikz}
\usetikzlibrary{shapes,arrows.meta, positioning}
\newtheorem*{theorem*}{Theorem}
\newtheorem{theorem}{Theorem}[section]
\newtheorem{lemma}[theorem]{Lemma}
\newtheorem{corollary}[theorem]{Corollary}
\newtheorem{proposition}[theorem]{Proposition}
\newtheorem{definition}[theorem]{Definition}
\newtheorem{example}[theorem]{Example}
\newtheorem{remark}[theorem]{Remark}
\newtheorem{question}[theorem]{Question}

\newcommand{\acknowledgments}{\section*{Acknowledgments}}

\begin{document}
\title[On $S$-$n$-absorbing ideals]
{On $S$-$n$-absorbing ideals}

\author [H. Baek] {Hyungtae Baek}
\address{(Baek) School of Mathematics,
Kyungpook National University, Daegu 41566,
Republic of Korea}
\email{htbaek5@gmail.com}

\author [H. S. Choi] {Hyun Seung Choi}
\address{(Choi) Department of Mathematics education,
Kyungpook National University, Daegu 41566,
Republic of Korea}
\email{hchoi21@knu.ac.kr}

\author [J. W. Lim] {Jung Wook Lim}
\address{(Lim) Department of Mathematics,
College of Natural Sciences,
Kyungpook National University, Daegu 41566,
Republic of Korea}
\email{jwlim@knu.ac.kr}

\thanks{Words and phrases: $S$-$n$-absorbing ideals, $n$-absorbing ideals, $S$-prime ideals, amalgamation of rings}

\thanks{$2020$ Mathematics Subject Classification: 13A15}

\begin{abstract}
Let $R$ be a commutative ring with identity,
$S$ a multiplicative subset of $R$ and
$I$ an ideal of $R$ disjoint from $S$.
In this paper, we introduce the notion of
an $S$-$n$-absorbing ideal which is a generalization of
both the $S$-prime ideals and $n$-absorbing ideals.
Moreover, we investigate the basic properties,
quotient extension, existence and amalgamation of $S$-$n$-absorbing ideals.
\end{abstract}

\maketitle

\section{Introduction}

Throughout this paper, $R$ is a commutative ring with identity,
$S$ is a (not necessarily saturated) multiplicative subset of $R$.
(For the sake of clarity,
we use $D$ instead of $R$ when $R$ is an integral domain.)
Also, $\mathbb{N}, \mathbb{N}_0, \mathbb{Z}$ and $\mathbb{Q}$ always denote
the set of natural numbers, nonnegative integers, integers and rational numbers, respectively.

In \cite{q80}, Querre showed that for an integrally closed domain $D$
and an integral ideal $B$ of $D[X]$ such that $B\cap D\neq0$,
the divisorial closure of $B$ equals the divisorial closure of $A_B[X]$,
where $A_B$ is the ideal of $D$ generated by coefficients of elements of $B$.
Anderson, Kwak and Zafrullah \cite{akz} introduced
``almost finitely generated ideals" and ``agreeable domains"
to study Querre's characterization of divisorial ideals.
In their 2002 paper, Anderson and Dumitrescu \cite{ad02}
extended these notions to arbitrary ideals of
a commutative ring with identity,
introducing $S$-finite ideals and $S$-Noetherian rings.
Namely, given a multiplicative subset $S$ and an ideal $I$ of a commutative ring $R$ with identity.
We say that $I$ is {\it $S$-finite}
if $sI\subseteq J\subseteq I$ for some $s\in S$ and a finitely generated ideal $J$ of $R$.
$R$ is {\it $S$-Noetherian} if every ideal of $R$ is $S$-finite.
Anderson and Dumitrescu also generalized several
well-known results of Noetherian rings to $S$-Noetherian rings,
including Eakin-Nagata theorem and Hilbert basis theorem \cite[Corollary 7, Proposition 10]{ad02}.

This initiated a series of papers that investigate $S$-Noetherianity on a specific class of rings,
and generalize several notions in multiplicative ideal theory to their $S$-versions.
For instance, Liu \cite{l07} examined when a generalized power series ring becomes $S$-Noetherian.
Lim and Oh \cite{lo14,lo15} studied
how $S$-Noetherian property behaves on amalgamated algebras along an ideal,
and composite ring extensions.
Kim, Kim and Lim introduced the notion of $S$-strong Mori domains extending
that of strong Mori domains \cite{kkl14},
while Hamed and Hizem \cite{hh} generalized GCD-domains to $S$-GCD-domains
which were later investigated by Anderson, Hamed and Zafrullah \cite{ahz}.

 In 2020, Hamed and Malek \cite{Hamed} defined $S$-prime ideals as a generalization of prime ideals.
Given a multiplicative subset $S$ and an ideal $I$ of a commutative ring $R$ disjoint from $S$,
$I$ is an {\it $S$-prime ideal} of $R$ if there exists $s\in S$ such that for each $a,b\in R$ such that $ab\in I$,
either $sa\in I$ or $sb\in I$. 
There are numerous different generalizations of prime ideals,
and among them, we focus on the $n$-absorbing ideals of Anderson and Badawi \cite{Anderson}.
For a natural number $n$ and an ideal $I$ of a commutative ring $R$,
we say that $I$ is an {\it $n$-absorbing ideal} of $R$
if for each $r_1,\dots, r_{n+1}\in R$ such that $r_1\cdots r_{n+1}\in I$,
there exists $1 \leq j \leq n+1$ such that $\prod_{\substack{1 \leq i \leq n+1 \\ i\neq j}}r_i\in I$.
The study of $n$-absorbing ideals quickly became an active research topic,
and several papers concerning properties of $n$-absorbing ideals were published (see the survey paper \cite{b23} and its reference list).
Considering these developments, we naturally ask whether $S$-prime ideals and $n$-absorbing ideals can be dealt with concurrently,
and in this paper, we define $S$-$n$-absorbing ideals of a commutative ring $R$ for a multiplicative subset $S$ and a natural number $n$.
It is easy to see that both an $S$-prime ideal, and
an $n$-absorbing ideal disjoint from $S$ are $S$-$n$-absorbing ideals for each $n\in\mathbb{N}$.
We show that several interesting properties of $S$-prime ideals and
$n$-absorbing ideals hold for $S$-$n$-absorbing ideals,
while there exist properties that hold for the former that do not extend smoothly to the latter.
To summarize, we present the following diagram.

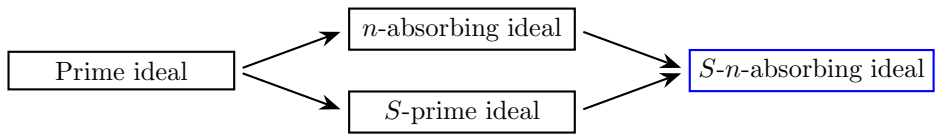
\begin{center}
\begin{tikzpicture}[scale=0.5, node distance=1.8cm, thick, nodes={draw, rectangle, align=center, minimum width=3cm, minimum height=0.5cm}, arrows=-{Stealth[scale=1.2]}, shorten >=3pt, shorten <=3pt]
    \node (prime) {Prime ideal};
    \node[right=1.5cm of prime, draw=white] (blank) {\color{white}{\tiny blank}};
    \node[above=0cm of blank] (nabsorbing) {$n$-absorbing ideal};
    \node[below=0cm of blank] (sprime) {$S$-prime ideal};
    \node[right=1.5cm of blank, draw=blue] (snabsorbing) { $S$-$n$-absorbing ideal};
    \draw (prime.east) -- (nabsorbing.west);
    \draw (prime.east) -- (sprime.west);
    \draw (nabsorbing.east) -- (snabsorbing.west);
    \draw (sprime.east) -- (snabsorbing.west);
\end{tikzpicture}
\end{center}

This paper consists of five sections including introduction.
In Section \ref{sec 2}, we study the basic properties of $S$-$n$-absorbing ideals.
In Section \ref{sec 3}, we analyze a theorem on the characterization of $S$-prime ideals,
and partially extend it to $S$-$n$-absorbing ideals.
In particular, we show that the statement of the theorem for
$S$-prime ideals can be extended to $S$-$n$-absorbing ideals
if $R$ is either a Laskerian ring or a locally divided ring.
More precisely, we show that for a multiplicative subset $S$ of $R$,
if $R$ is an $S$-Laskerian ring,
then $IR_S$ is an $n$-absorbing ideal of $R_S$ if and only if
$I$ is an $S$-$n$-absorbing ideal of $R$ for each ideal $I$ of $R$ disjoint from $S$ (Theorem \ref{one-to-one cor(Laskerian)}).
Section \ref{sec 4} is devoted to the study of the minimal positive integer $n$ that
makes an ideal $I$ of a ring $R$ $S$-$n$-absorbing,
and the set of all such $n$, collected from every ideal of $R$ disjoint from $S$.
In this direction, we seek the conditions on a ring such that $\omega_S(R) \subseteq \mathbb{N}$.
In Section \ref{sec 5}, the final section of this paper,
focuses on $S$-$n$-absorbing ideals of the amalgamation of rings.
Specifically, it is shown that when $S$-$n$-absorbing ideals are extended in amalgamation,
the resulting ideals are not necessarily $S$-$n$-absorbing ideals.
On the other hand, a sufficient condition of an ideal primary to the maximal ideal of a local ring
being $S$-$n$-absorbing ideal when extended to an amalgamation of rings is given
(Proposition \ref{local amalgam}).
Several examples are given to illustrate the theorems of this paper.
For any undefined terms, we refer the reader to \cite{G}.

\section{Basic properties of $S$-$n$-absorbing ideals}\label{sec 2}

In this section, we define and consider some basic properties of $S$-$n$-absorbing ideals.
More precisely, we investigate when well-known ideals
become $S$-$n$-absorbing ideals,
and we also examine the relationship between $S$-$n$-absorbing ideals and minimal prime ideals.

\begin{definition}
{\rm
Let $R$ be a commutative ring with identity,
$S$ a multiplicative subset of $R$ and
$I$ an ideal disjoint from $S$.
Then $I$ is an {\it $S$-$n$-absorbing ideal} if
there exists $s \in S$ such that
$r_1,\dots, r_{n+1} \in R$ with $r_1 \cdots r_{n+1} \in I$ implies
$s \prod_{\substack{1 \leq i \leq n+1 \\ i \neq j}}r_i \in I$ for some $1 \leq j \leq n+1$.
In this case, we say that $I$ is {\it associated} to $s$.
}
\end{definition}

\begin{example}
{\rm
(1) If $S$ consists of units of $R$, then the $n$-absorbing ideals and $S$-$n$-absorbing ideals coincide.

(2) Every $S$-prime ideal is an $S$-$n$-absorbing ideal for all $n \in \mathbb{N}$.

(3) Every $n$-absorbing ideal is an $S$-$n$-absorbing ideal for any multiplicative subset $S$ of $R$.

(4) In general, the converse of (2) and (3) do not hold.
Let $S = \{4^n \,|\, n \in \mathbb{N}\}$.
Then $S$ is a multiplicative subset of $\mathbb{Z}[X]$.
Consider the ideal $I := 2X^2\mathbb{Z}[X]$.
Let $f,g,h \in \mathbb{Z}[X]$ with $fgh \in I$.
Since $X^2\mathbb{Z}[X]$ is a $X\mathbb{Z}[X]$-primary ideal, $X^2$ divides one of the $fg, gh, fh$.
Therefore at least one of $4fg,4gh,4fh$ must be an element of $I$.
This implies that $I$ is an $S$-$2$-absorbing ideal of $R$.
However, $I$ is not an $S$-prime ideal of $R$.
Indeed, $2X \cdot X \in I$, but neither $4^nX$ nor $4^n(2X)$ belong to $I$ for each $n\in\mathbb{N}$.
Moreover, $I$ is not a $2$-absorbing ideal of $R$,
because $2\cdot X\cdot X \in I$, but neither $2X$ nor $X^2$ belong to $I$.
}
\end{example}

\begin{proposition}\label{basic property}
Let $R$ be a commutative ring with identity and
let $S,S_1,\dots,S_m$ be multiplicative subsets of $R$.
Then the following assertions hold.
\begin{enumerate}
\item[(1)]
Let $I$ be an ideal of $R$ disjoint from $S$ and
let $J$ be an ideal of $R$ such that $J \cap S \neq \emptyset$.
If $I$ is an $S$-$n$-absorbing ideal of $R$,
then $IJ$ is an $S$-$n$-absorbing ideal of $R$.
\item[(2)]
Let $T$ be a commutative ring with identity containing $R$ and
let $J$ be an ideal of $T$ disjoint from $S$.
If $J$ is an $S$-$n$-absorbing ideal of $T$,
then $J \cap R$ is an $S$-$n$-absorbing ideal of $R$.
\item[(3)]
For each $1 \leq j \leq m$,
let $I_j$ be an ideal of $R$ disjoint from $S_j$.
If for each $1 \leq k \leq m$,
$I_k$ is an $S_k$-$n_k$-absorbing ideal of $R$,
then $I_1 \cap \cdots \cap I_m$ is an $S$-$n$-absorbing ideal of $R$,
where $S : = S_1 \cdots S_m = \{s_1\cdots s_m \,|\, s_k \in S_k {\rm \ for \ all \ } 1 \leq k \leq m \}$ and
$n = n_1 + \cdots + n_m$.
\item[(4)]
Let $I_1,\dots,I_m$ be ideals of $R$ disjoint from $S$.
If for each $1 \leq i \leq m$,
$I_i$ is an $S$-$n_i$-absorbing ideal of $R$,
then $I_1 \cap \cdots \cap I_m$ is an $S$-$n$-absorbing ideal of $R$,
where $n = n_1 + \cdots + n_m$.
\item[(5)]
Let $I$ be an ideal of $R$ disjoint from $S$.
If $I$ is an $S$-$n$-absorbing ideal of $R$ and
associated to $s$ for some $s \in S$,
then $\sqrt{I}$ is an $S$-$n$-absorbing ideal of $R$, and
for each $a \in \sqrt{I}$, there exists $t \in S$ such that $ta^n \in I$.
\end{enumerate}
\end{proposition}

\begin{proof}
(1) Suppose that $I$ is an $S$-$n$-absorbing ideal of $R$ which is
associated to for some $s \in S$.
As $J \cap S \neq \emptyset$,
we can pick an element $t \in J \cap S$.
Let $r_1,\dots,r_{n+1} \in R$ with $r_1\cdots r_{n+1} \in IJ$.
Then $r_1\cdots r_{n+1} \in I$,
so we may assume that $sr_1\cdots r_n \in I$.
Then $str_1\cdots r_n \in IJ$.
Thus $IJ$ is an $S$-$n$-absorbing ideal of $R$.

(2) Suppose that $J$ is an $S$-$n$-absorbing ideal of $T$.
Let $r_1, \dots, r_{n+1} \in R$ with $r_1\cdots r_{n+1} \in J \cap R$.
Then we may assume that $sr_1\cdots r_n \in J$ for some $s \in S$.
Thus $sr_1 \cdots r_n \in J\cap R$,
which means that $J \cap R$ is an $S$-$n$-absorbing ideal of $R$.

(3) For the sake of simplicity, we assume that $m=2$.
Let $S = S_1S_2$, and choose $r_1, \dots, r_{n+1} \in R$ with
$r_1 \cdots r_{n+1} \in I_1 \cap I_2$.
Since $I_1$ is an $S_1$-$n_1$-absorbing ideal of $R$ and
$I_2$ is an $S_2$-$n_2$-absorbing ideal of $R$,
there exist elements $s_1 \in S_1$, $s_2 \in S_2$ and subsets
$A_1 := \{r_{11},\dots, r_{1n_1}\}, A_2 := \{r_{21},\dots, r_{2n_2}\}
\subseteq \{r_1,\dots,r_{n+1}\}$ such that
$s_1 \prod_{1 \leq k \leq n_1}r_{1k} \in I_1$ and
$s_2 \prod_{1 \leq \ell \leq n_2}r_{2\ell} \in I_2$.
First we suppose that $A_1 \subseteq A_2$.
Then $s_1s_2 \prod_{1 \leq k \leq n_2}r_{2k} \in I_1 \cap I_2$,
so $I_1\cap I_2$ is an $S$-$n_2$-absorbing ideal of $R$.
Since $n_2 \leq n$, $I_1 \cap I_2$ is an $S$-$n$-absorbing ideal of $R$.
On the other hand, consider the case $A_1 \not\subseteq A_2$.
Then $s_1s_2 \big(\prod_{1 \leq k \leq n_1}r_{1k}\big) \big(\prod_{\alpha \in A_2 \setminus A_1} \alpha \big)  \in I_1 \cap I_2$.
Thus $I_1 \cap I_2$ is an $S$-$n$-absorbing ideal of $R$
since $n_1 + |A_2 \setminus A_1| \leq n$.

(4) This result follows directly from (3) since $S^n \subseteq S$.

(5) Suppose that $I$ is an $S$-$n$-absorbing ideal of $R$ and
associated to $s$.
Let $a_1,\dots,a_{n+1} \in R$ with $a_1 \cdots a_{n+1} \in \sqrt{I}$.
Then $a_1^m \cdots a_{n+1}^m \in I$ for some $m \in \mathbb{N}$.
Hence we may assume that $sa_1^m \cdots a_n^m \in I$,
so $(sa_1 \cdots a_n)^m \in I$.
Hence $sa_1 \cdots a_n \in \sqrt{I}$,
and thus $\sqrt{I}$ is an $S$-$n$-absorbing ideal of $R$.
For the remainder argument,
let $a \in \sqrt{I}$.
Then $a^m \in I$ for some $m \in \mathbb{N}$.
If $m \leq n$, then we are done.
Now, suppose that $m \geq n$.
Then there exists $\alpha \in \mathbb{N}$ such that
$s^{\alpha} a^n \in I$.
Since $s^{\alpha} \in S$,
the proof is done.
\end{proof}

Let $R$ be a commutative ring with identity and
let $I_1,\dots,I_m$ be ideals of $R$.
Recall that $I_1,\dots,I_m$ are {\it comaximal}
if $I_{\alpha} + I_{\beta} = R$ for any $1 \leq \alpha \neq \beta \leq m$.
Hence we have if $I_1,\dots,I_m$ are comaximal, then $I_1\cap \cdots \cap I_m = I_1\cdots I_m$.

\begin{corollary}
Let $R$ be a commutative ring with identity,
$S$ a multiplicative subset of $R$ and
$P_1,\dots,P_n$ be ideals of $R$ disjoint from $S$
which are pairwise comaximal.
If $P_1,\dots, P_n$ are $S$-prime ideals of $R$,
then $P_1\cdots P_n$ is an $S$-$n$-absorbing ideal of $R$.
\end{corollary}

\begin{proof}
This result follows directly from Proposition \ref{basic property}(4) and
the fact that an $S$-prime ideal is a $1$-absorbing ideal.
\end{proof}

Let $R$ and $T$ be commutative rings with identity,
let $S$ a multiplicative subset of $R$ and
$\varphi : R \to T$ a homomorphism.
Then it is easy to see that $\varphi(S)$ is a multiplicative subset of $\varphi(R)$.

\begin{proposition}\label{homo cor}
Let $R$ and $T$ be commutative rings with identity,
$S$ a multiplicative subset of $R$ and
$\varphi : R \to T$ a homomorphism.
\begin{enumerate}
\item[(1)]
Let $I$ be an ideal of $R$ disjoint from $S$ with $\ker(\varphi) \subseteq I$.
Then $I$ is an $S$-$n$-absorbing ideal of $R$ if and only if
$\varphi(I)$ is a $\varphi(S)$-$n$-absorbing ideal of ${\rm Im}(\varphi)$.
\item[(2)]
Let $J$ be an ideal of ${\rm Im}(\varphi)$ disjoint from $\varphi(S)$.
Then $J$ is a $\varphi(S)$-$n$-absorbing ideal of ${\rm Im}(\varphi)$ if and only if
$\varphi^{-1}(J)$ is an $S$-$n$-absorbing ideal of $R$.
\end{enumerate}
\end{proposition}

\begin{proof}
(1) Suppose that $I$ is an $S$-$n$-absorbing ideal of $R$.
First, we claim that $\varphi(I) \cap \varphi(S) = \emptyset$.
Suppose to the contrary that
there exists an element $i \in I$ such that
$\varphi(i) \in \varphi(S)$.
Hence $i-s \in \ker(\varphi) \subseteq I$ for some $s \in S$,
so $s \in I$.
This is a contradiction since $I \cap S = \emptyset$.
Therefore $\varphi(I)$ is disjoint from $\varphi(S)$.
Let $\varphi(r_1),\dots,\varphi(r_{n+1}) \in {\rm Im}(\varphi)$ such that
$\varphi(r_1) \cdots \varphi(r_{n+1}) \in \varphi(I)$.
Then there exists an element $i \in I$ such that
$r_1 \cdots r_{n+1} - i \in \ker(\varphi) \subseteq I$.
Hence $r_1 \cdots r_{n+1} \in I$.
Since $I$ is an $S$-$n$-absorbing ideal of $R$,
we may assume that there exists an element $s \in S$ such that
$s r_1 \cdots r_n \in I$.
Hence $\varphi(s) \varphi(r_1) \cdots \varphi(r_n) \in \varphi(I)$,
which means that $\varphi(I)$ is a $\varphi(S)$-$n$-absorbing ideal of ${\rm Im}(\varphi)$.
For the converse, suppose that $\varphi(I)$ is a $\varphi(S)$-$n$-absorbing ideal of ${\rm Im}(\varphi)$.
It is clear that $I \cap S = \emptyset$.
Let $r_1,\dots,r_{n+1} \in R$ with $r_1 \cdots r_{n+1} \in I$.
Then $\varphi(r_1) \cdots \varphi(r_{n+1}) \in \varphi(I)$,
so we may assume that there exists an element $s \in S$ such that
$\varphi(sr_1 \cdots r_n) = \varphi(s)\varphi(r_1) \cdots \varphi(r_n) \in \varphi(I)$.
It follows that there exists an element $i \in I$ such that
$sr_1 \cdots r_n - i \in \ker(\varphi) \subseteq I$.
Hence $sr_1 \cdots r_n \in I$.
Thus $I$ is an $S$-$n$-absorbing ideal of $R$.

(2) Note that $J \cap \varphi(S) = \emptyset$ if and only if $\varphi^{-1}(J) \cap S = \emptyset$.
Suppose that $J$ is a $\varphi(S)$-$n$-absorbing ideal of ${\rm Im}(\varphi)$.
Let $r_1,\dots, r_{n+1} \in R$ with $r_1\cdots r_{n+1} \in \varphi^{-1}(J)$.
Then $\varphi(r_1) \cdots \varphi(r_{n+1}) = \varphi(r_1\cdots r_{n+1}) \in J$,
so we may assume that $\varphi(sr_1\cdots r_n) = \varphi(s)\varphi(r_1) \cdots \varphi(r_n) \in J$
for some $s \in S$.
Thus $sr_1\cdots r_n \in \varphi^{-1}(J)$.
Consequently, $\varphi^{-1}(J)$ is an $S$-$n$-absorbing ideal of $R$.
The reverse assertion follows directly from (1).
\end{proof}

By Proposition \ref{homo cor},
we obtain

\begin{corollary}
Let $R$ be a commutative ring with identity,
$S$ a multiplicative subset of $R$ and
$\varphi : R \to T$ a homomorphism with $\ker(\varphi) \cap S = \emptyset$.
Then there is a one-to-one order-preserving correspondence between
the $S$-$n$-absorbing ideals of $R$ containing $\ker(\varphi)$ and
the $\varphi(S)$-$n$-absorbing ideals of $\varphi(R)$.
\end{corollary}

Let $R$ be a commutative ring with identity,
$I$ an ideal of $R$ and
$S$ a multiplicative subset of $R$.
Then $S/I :=\{s+I \,|\, s \in S\}$ is a multiplicative subset of $R/I$.

\begin{corollary}\label{factor cor}
Let $R$ be a commutative ring with identity,
$S$ a multiplicative subset of $R$,
and $I, J$ ideal of $R$ such that $I \subseteq J$
and $J \cap S = \emptyset$.
Then $J$ is an $S$-$n$-absorbing ideal of $R$ if and only if
$J/I$ is an $S/I$-$n$-absorbing ideal of $R/I$.
\end{corollary}

Let $R$ be a commutative ring with identity,
$S$ a multiplicative subset of $R$ and
$I$ an ideal of $R$.
Recall that $I:s = \{r \in R \,|\, sr \in I\}$ is an ideal of $R$ and
which is called by a {\it colon ideal}.
From now on, we investigate the behavior of the colon ideal $I:s$
when $I$ is an $S$-$n$-absorbing ideal.
The following result facilitates the study of $S$-$n$-absorbing ideals
by applying the properties of $n$-absorbing ideals
and is also a useful fact in this paper.

\begin{proposition}\label{colon}
Let $R$ be a commutative ring with identity,
$S$ a multiplicative subset of $R$ and
$I$ an ideal of $R$ disjoint from $S$.
Then $I$ is an $S$-$n$-absorbing ideal of $R$ if and only if
$I:s$ is an $n$-absorbing ideal of $R$ for some $s \in S$.
\end{proposition}

\begin{proof}
Suppose that $I$ is an $S$-$n$-absorbing ideal of $R$
which is associated to $t \in S$.
We claim that $I : t^{n+1}$ is an $n$-absorbing ideal of $R$.
Let $r_1,\dots, r_{n+1} \in R$ with $r_1\cdots r_{n+1} \in I:t^{n+1}$.
Then $(tr_1)(tr_2)\cdots (tr_{n+1}) \in I$,
which means that $t \prod_{\substack{1 \leq i \leq n+1 \\ i \neq j}}tr_i \in I$ for some $1 \leq j \leq n+1$.
This implies that $\prod_{\substack{1 \leq i \leq n+1 \\ i \neq j}}r_i \in  I : t^{n+1}$.
Thus $I:s$ is an $n$-absorbing ideal of $R$, where $s=t^{n+1}\in S$.
Conversely, suppose that $I:s$ is an $n$-absorbing ideal of $R$ for some $s\in S$ and
let $r_{1},\dots, r_{n+1}\in R$ such that $r_{1}\cdots r_{n+1}\in I$.
Since $r_{1}\cdots r_{n+1}\in I:s$,
$\prod_{\substack{1 \leq i \leq n+1 \\ i \neq j}}r_{i}\in I:s$ for some $1 \leq j \leq n+1$.
In other words, $s\prod_{\substack{1 \leq i \leq n+1 \\ i \neq j}}r_{i}\in I$ for some $1 \leq j \leq n+1$.
Thus $I$ is an $S$-$n$-absorbing ideal of $R$.
\end{proof}

Let $R$ be a commutative ring with identity and
let $I$ be an ideal of $R$.
Recall that a {\it minimal prime ideal} of $I$ is
a prime ideal of $R$ minimal among the ones containing $I$.

\begin{corollary}
Let $R$ be a commutative ring with identity,
$S$ a multiplicative subset of $R$ and
$I$ an ideal of $R$ disjoint from $S$.
If $I$ is an $S$-$n$-absorbing ideal of $R$,
then there are at most $n$ minimal prime ideals of $I$
disjoint from $S$.
\end{corollary}

\begin{proof}
Suppose that $I$ is an $S$-$n$-absorbing ideal of $R$.
Then $I:s$ is an $n$-absorbing ideal of $R$ for some $s \in S$ by Proposition \ref{colon}.
Hence there are at most $n$ minimal prime ideals of $I:s$ \cite[Theorem 2.5]{Anderson}.
As a minimal prime of $I$ disjoint from $S$ is a minimal prime of $I : s$,
the number of the minimal prime ideals of $I:s$ is greater than
the number of the minimal prime ideals of $I$ disjoint from $S$.
Hence the proof is done.
\end{proof}

Let $R$ be a commutative ring with identity,
$S$ a multiplicative subset of $R$ and
$I$ an ideal of $R$.
In a manner similar to the minimal prime ideal, we can define a minimal $S$-$n$-absorbing ideal of $I$ as follows:
$I$ is a {\it minimal $S$-$n$-absorbing ideal over $J$} if
$I$ is a minimal element of the set of the $S$-$n$-absorbing ideals containing $J$.

In \cite{Hamed}, the authors defined the following concept:
a multiplicative subset $S$ of $R$ is a {\it strongly multiplicative subset} if
for each subset $\{s_{\alpha} \,|\, \alpha \in \Lambda \}$ of $S$,
$(\bigcap_{\alpha \in \Lambda} s_{\alpha}R) \cap S \neq \emptyset$.
Note that every finite multiplicative subset of $R$ is
a strongly multiplicative subset of $R$,
and every multiplicative subset which is not anti-Archimedean is not a strongly multiplicative subset of $R$
\cite[Example 5]{Hamed}.

From this point onward,
we explore the existence of a minimal $S$-$n$-absorbing ideal.
For this purpose, we need the following lemma.

\begin{lemma}\label{intersection}
Let $R$ be a commutative ring with identity,
$S$ a strongly multiplicative subset of $R$ and
$\{I_{\alpha} \,|\, \alpha \in \Lambda\}$ a chain of $S$-$n$-absorbing ideals of $R$.
Then $\bigcap_{\alpha \in \Lambda} I_{\alpha}$ is an $S$-$n$-absorbing ideal of $R$.
\end{lemma}

\begin{proof}
Suppose that for each $\alpha \in \Lambda$,
$I_{\alpha}$ is an $S$-$n$-absorbing ideal of $R$
which is associated to $s_{\alpha} \in S$.
Consider the subset $\{s_{\alpha} \,|\, \alpha \in \Lambda\}$ of $S$.
Since $S$ is a strongly multiplicative subset of $R$,
we can pick an element $t \in (\bigcap_{\alpha \in \Lambda} s_{\alpha}R) \cap S$.
Now, suppose that $r_1,\dots, r_{n+1} \in R$ with $r_1 \cdots r_{n+1} \in \bigcap_{\alpha \in \Lambda} I_{\alpha}$.
Suppose to the contrary that
$t\prod_{\substack{1 \leq i \leq n+1 \\ i \neq j}}r_i \notin
\bigcap_{\alpha \in \Lambda} I_{\alpha}$ for all $1\leq j \leq n+1$.
Then there exist elements $\alpha_{1},\dots, \alpha_{n+1}\in\Lambda$
such that $t\prod_{\substack{1 \leq i \leq n+1 \\ i \neq j}}r_i \notin I_{\alpha_j}$.
Since $\{I_{\alpha} \,|\, \alpha \in \Lambda\}$ is a chain,
there exists an integer $1 \leq m \leq n+1$ such that
$ I_{\alpha_m} = \bigcap_{i=1}^{n+1} I_{\alpha_i}$,
so $t\prod_{\substack{1 \leq i \leq n+1 \\ i \neq j}}r_i \notin I_{\alpha_m}$ for all $1 \leq j \leq n+1$;
that is, $s_{\alpha_m}\prod_{\substack{1 \leq i \leq n+1 \\ i \neq j}}r_i \notin I_{\alpha_m}$ for all $1 \leq j \leq n+1$.
This contradicts the fact that $I_{\alpha_m}$ is an $S$-$n$-absorbing ideal and
$I_{\alpha_m}$ is associated to $s_{\alpha_m}$.
Thus there exists a positive integer $1 \leq j \leq n+1$ such that
$t\prod_{\substack{1 \leq i \leq n+1 \\ i \neq j}}r_i \in
\bigcap_{\alpha \in \Lambda} I_{\alpha}$.
Consequently, $\bigcap_{\alpha \in \Lambda} I_{\alpha}$ is an $S$-$n$-absorbing ideal of $R$.
\end{proof}

We conclude this section with the following result.

\begin{proposition}
Let $R$ be a commutative ring with identity and
let $S$ be a strongly multiplicative subset of $R$.
Then each ideal $I$ of $R$ disjoint from $S$ is
contained in minimal $S$-$n$-absorbing ideal over $I$.
\end{proposition}

\begin{proof}
Let $I$ be an ideal of $R$ disjoint from $S$ and
let $\mathfrak{A}$ be the set of the $S$-$n$-absorbing ideals of $R$ containing $I$.
Note that there is a prime ideal containing $I$ disjoint from $S$, say $P$.
This implies that $\mathfrak{A} \neq \emptyset$.
Also, by Lemma \ref{intersection}, $(\mathfrak{A},\supseteq)$ has a maximal element.
Thus the proof is done.
\end{proof}

\section{$S$-$n$-absorbing ideals of quotient rings}\label{sec 3}

Motivated by the following theorem, we examine the quotient extension of $S$-$n$-absorbing ideals in this section.

\begin{theorem*}
{\rm (\cite[Remark 1]{Hamed})}
Let $R$ be a commutative ring with identity,
$S$ a multiplicative subset of $R$ consisting of nonzerodivisors and
$P$ an ideal of $R$ disjoint with $S$.
Then the following assertions are equivalent:
\begin{enumerate}
\item[(1)]
$P$ is an $S$-prime ideal of $R$.
\item[(2)]
$P:s$ is a prime ideal of $R$ for some $s \in S$.
\item[(3)]
$PR_S$ is a prime ideal of $R_S$ and
$PR_S \cap R = P:s$ for some $s \in S$.
\end{enumerate}
\end{theorem*}

Consider the following statement.
\begin{center}
(3$'$) $PR_S$ is a prime ideal of $R_S$.
\end{center}
We investigate whether the theorem above can be generalized to $S$-$n$-absorbing ideals.
Moreover, we check whether the condition $(3)$ can be replaced to (3$'$).

The generalization of the equivalence of (1) and (2) of the previous theorem to $S$-$n$-absorbing ideals
was proved in Proposition \ref{colon}.

Now, we investigate whether the assertions (1) and (3$'$)
are equivalent for $S$-$n$-absorbing ideals.

\begin{proposition}\label{quotient extension}
Let $R$ be a commutative ring with identity,
$S$ a multiplicative subset of $R$ and
$I$ an ideal of $R$ disjoint from $S$.
If $I$ is an $S$-$n$-absorbing ideal of $R$,
then $IR_S$ is an $n$-absorbing ideal of $R_S$.
\end{proposition}

\begin{proof}
Suppose that $I$ is an $S$-$n$-absorbing ideal of $R$.
Let $a_1,\dots, a_{n+1} \in R_S$ such that
$a_1\cdots a_{n+1} \in IR_S$.
Since for each $1 \leq i \leq n+1$,
$a_i = \frac{r_i}{s_i}$ for some $r_i \in R$ and $s_i \in S$,
we obtain that $\frac{r_1\cdots r_{n+1}}{s_1\cdots s_{n+1}} \in IR_S$.
Hence there exist elements $s \in S$ such that
$sr_1 \cdots r_{n+1} \in I$.
Set $s = r_{n+2}$.
As $I$ is an $S$-$n$-absorbing ideal of $R$,
there exist $u\in S$ and distinct integers $1 \leq i,j \leq n+2$ such that
$u \prod_{\substack{1 \leq \alpha \leq n+2\\ \alpha \neq i,j}}r_{\alpha} \in I$.
It is routine to check that
$\prod_{\substack{1 \leq \alpha \leq n+1 \\ \alpha \neq \beta}} \frac{r_{\alpha}}{s_{\alpha}} \in IR_S$
for some $1 \leq \beta \leq n+1$.
Thus $IR_S$ is an $n$-absorbing ideal of $R_S$.
\end{proof}

The next example shows that
for each $n \geq 1$, the converse of Proposition \ref{quotient extension} does not necessarily hold.

\begin{example}\label{counter example of converse quotient}
{\rm
Let $R=\mathbb{Z}+X\mathbb{Q}[X]$ and
let $n$ be a fixed positive integer.
Choose a prime number $q$,
and let $I=X^n((X-q)\mathbb{Q}[X]\cap R)$,
$M=X\mathbb{Q}[X]$ and $S=R\setminus M$.
Then $I$ is an ideal of $R$, $M$ is a prime ideal of $R$, and $I\subseteq M$.
Hence $S$ is a multiplicative subset of $R$ disjoint from $I$.
We then claim that $I$ is not an $S$-$n$-absorbing ideal of $R$.
Indeed, suppose that $I$ is an $S$-$n$-absorbing ideal.
As $X^n(X-q)\in I$,
there exists $s\in S$ such that
$sX^n\in I$ or $sX^{n-1}(X-q)\in I$.
Assume that $sX^{n-1}(X-q)\in I$.
Then $sX^{n-1}(X-q)=X^n(X-q)u$
for some $u\in \mathbb{Q}[X]$ such that $qu(0)\in \mathbb{Z}$,
and hence $s=Xu \in M$, a contradiction.
Therefore $sX^n\in I$, so $sX^n=X^n(X-q)h$ for some $h\in \mathbb{Q}[X]$
such that $qh(0)\in\mathbb{Z}\setminus \{0\}$.
Now set $r=qh(0)$ and choose a prime number $p$ that does not divide $r$,
and let $a_1=\cdots=a_n=\frac{1}{p}X$ and $a_{n+1}=p^n(X-q)$.
Then $a_1,\dots,a_{n+1}$ are elements of $R$ such that $a_1\cdots a_{n+1}\in I$.
Moreover, if $s \frac{1}{p^n}X^n = sa_1\cdots a_n \in I$,
then $s\frac{1}{p^n}X^n=X^n(X-q)l$ for some $l\in\mathbb{Q}[X]$ such that $ql(0)\in \mathbb{Z}$.
This implies that $(X-q)h= s=p^n(X-q)l$, and hence $p^nl=h$.
Therefore we have $\frac{r}{p^n}=\frac{qh(0)}{p^n}=ql(0)\in\mathbb{Z}$,
which contradicts our choice of $p$.
Hence $sa_1\cdots a_n\not\in I$.
On the other hand,
for any $1 \leq i \leq n$,
$pX^{n-1}(X-q)^2 h = s \prod_{\substack{1\leq j \leq n+1 \\ j \neq i}} a_j \not\in I$
since $pX^{n-1}(X-q)^2 h$ cannot be divided by $X^n$ in $R$.
Hence $I$ is not an $S$-$n$-absorbing ideal of $R$.
Finally, we have $R_{S}= \mathbb{Q}[X]_{X\mathbb{Q}[X]}$,
and thus $IR_{S}=X^n\mathbb{Q}[X]_{X\mathbb{Q}[X]}$ is an $n$-absorbing ideal of $R_{S}$
\cite[Example 5.6(a)]{Anderson}.
}
\end{example}

Let $R$ be a commutative ring with identity and
let $I$ be an ideal of $R$.
Recall that $I$ is said to be a {\it primary ideal} if
for each $a,b\in R$ with $ab\in I$, either $a\in I$ or $b\in\sqrt{I}$.
A primary ideal that contains a power of its radical
is said to be a {strongly primary ideal}.
Also, recall that $R$ is {\it Laskerian}
(respectively, {\it strongly Laskerian})
if for each proper ideal $I$ of $R$,
$I$ can be written as an intersection of finitely many primary ideals
(respectively, strongly primary ideals) of $R$.

S. Visweswaran \cite{v22} introduced the notion of
(strongly) $S$-primary ideals in order to generalize
(strongly) Laskerian rings.
Namely, given a multiplicative subset $S$ of $R$,
an ideal $I$ of $R$ is an {\it $S$-primary ideal} if
$I\cap S=\emptyset$ and there exists $s\in S$ such that
whenever $ab\in I$ for some $a,b\in R$, either $sa\in I$ or $sb\in\sqrt{I}$.
An $S$-primary ideal is a {\it strongly $S$-primary ideal} if
$t(\sqrt{I})^n\subseteq I$ for some $n\in\mathbb{N}$
and $t\in S$.
We say $R$ is {\it $S$-Laskerian}
(respectively, {\it strongly $S$-Laskerian})
if for each proper ideal $I$ of $R$ disjoint from $S$,
there exists an element $s \in S$ such that
$I:s$ can be written
as an intersection of finitely many $S$-primary ideals
(respectively, strongly $S$-primary ideals) of $R$.
Clearly, every primary ideal (respectively, strongly primary ideal) of $R$ disjoint from $S$
is $S$-primary (respectively, strongly $S$-primary),
while every Laskerian (respectively, strongly Laskerian) ring is $S$-Lakserian
(respectively, strongly $S$-Laskerian) for each multiplicative subset $S$ of $R$.

Let $R=\mathbb{Z}+X\mathbb{Q}[X]$.
Then $R$ is a  Pr{\"u}fer domain with Krull dimension $2$ \cite[Corollary 1.1.9]{fhp},
so $R$ is not a Laskerian ring \cite[Exercise 37.9]{G}.
Hence the converse of Proposition \ref{quotient extension} does not necessarily hold
when $R$ is not a Laskerian ring by Example \ref{counter example of converse quotient}.
Therefore we naturally ask whether 
the converse of proposition \ref{quotient extension} holds when $R$ is a Laskerian ring.
In fact, we show that this question has an affirmative answer for $S$-Laskerian rings.

Recall that if $I$ is a proper ideal of $R$ disjoint from $S$,
then the ideal $\{x\in R\mid xs\in I \textnormal{ for some }s\in S\}$ of $R$
is called the {\it $S$-saturation} of $I$ (or the {\it contraction} of $IR_{S}$ to $R$),
denoted by ${\rm Sat}_S(I)$.
This is the preimage of the ideal $IR_{S}$ to $R$
via the canonical homomorphism $\varphi_s : R\to R_{S}$
given by $\varphi_s(r)=\frac{rs}{s}$
for each $r\in R$ and fixed element $s \in S$.
As in the motivating theorem,
sometimes the notation $IR_{S}\cap R$ is used to indicate ${\rm Sat}_S(I)$.
It is easy to see that $I\subseteq {\rm Sat}_S(I)$.
The following result is implicitly suggested in \cite{v22},
and we record it for convenience.

\begin{lemma}\label{v31}
Let $R$ be a commutative ring with identity,
$S$ a multiplicative subset of $R$ and
$I$ an ideal of $R$ disjoint from $S$.
Then the following are equivalent.
\begin{enumerate}
\item[(1)] $I:s$ is a finite intersection of $S$-primary ideals of $R$ for some $s\in S$.
\item[(2)] $I:s'$ is a finite intersection of primary ideals of $R$ for some $s'\in S$.
\item[(3)] $IR_S$ is a finite intersection of primary ideals of $R_S$,
and ${\rm Sat}_S(I)=I:t$ for some $t\in S$. 
\end{enumerate}
\end{lemma}

\begin{proof}
The equivalence of (1) and (3) follows from  \cite[Lemma 3.1]{v22}.
 
(1) $\Rightarrow$ (2)
Suppose that $I:s=Q_1\cap\cdots \cap Q_r$ for some $s\in S$
and $S$-primary ideals $Q_1,\dots, Q_r$ of $R$.
Then there exist elements $s_1,\dots, s_r$ of $S$
such that $Q_i:s_i$ is a primary ideal of $R$
for each $1 \leq i \leq r$
\cite[Theorem 2.7]{v22}.
Let $s'=ss_1\cdots s_r$.
Then $s'\in S$, and $Q_i:s_1\cdots s_r=(Q_i:s_i):\prod_{\substack{1 \leq j \leq r \\ j\neq i}}s_j=Q_i:s_i$
for each $1 \leq i \leq r$.
Therefore $I:s'=\bigcap_{i=1}^{r}(Q_i:s_1\cdots s_r)=\bigcap_{i=1}^{r}(Q_i:s_i)$
is a finite intersection of primary ideals of $R$,
and thus (2) follows.

(2) $\Rightarrow$ (1)
Suppose that $I:s'$ is a finite intersection of primary ideals
$Q_1,\dots, Q_m$ for some $s'\in S$.
Then at least one of $Q_1,\dots, Q_m$, say $Q_1$,
must be disjoint from $S$
since $I$ is disjoint from $S$.
Hence by reordering $Q_1,\dots, Q_m$,
we may assume that there exists $1 \leq \ell \leq m$
such that $Q_i$ is disjoint from $S$
if and only if $i \leq \ell$. It follows that for each integer $i$
such that $\ell< i \leq m$,
there exists $s_i\in Q_i\cap S$.
Set $s=s's_{l+1}\cdots s_{m}$.
Then $s\in S$, and we have $I:s=Q_1\cap\cdots\cap Q_{\ell}$.
Since $Q_1,\dots, Q_{\ell}$ are $S$-primary, (1) follows.
\end{proof}

\begin{theorem}\label{one-to-one cor(Laskerian)}
Let $R$ be a commutative ring with identity,
$S$ a multiplicative subset of $R$ and
$I$ an ideal of $R$ disjoint from $S$.
Suppose that $R$ is an $S$-Laskerian ring.
Then $IR_{S}$ is an $n$-absorbing ideal of $R_{S}$
if and only if $I$ is an $S$-$n$-absorbing ideal of $R$.
\end{theorem}

\begin{proof}
Suppose that $IR_{S}$ is an $n$-absorbing ideal of $R_{S}$.
Then ${\rm Sat}_S(I)$ is an $n$-absorbing ideal of $R$ \cite[Theorem 4.2(a)]{Anderson}.
Since $R$ is an $S$-Laskerian ring, we have ${\rm Sat}_S(I)=I:s$ for some $s\in S$ by Lemma \ref{v31}. 
Hence by Proposition \ref{colon},
$I$ is an $S$-$n$-absorbing ideal of $R$.
The converse follows from Proposition \ref{quotient extension}.
\end{proof}

Now, we seek a sufficient condition for the equivalence of (1) and (3) of the motivating theorem.
Let $R$ be a commutative ring with identity and
let $I$ be an ideal of $R$.
Recall that $I$ is {\it divided}
if for each element $a$ of $R$, either $a\in I$ or $I\subsetneq aR$.
By $R$ is a {\it divided ring}
we mean in which every prime ideal is divided.
Finally, $R$ is {\it locally divided}
if $R_{M}$ is a divided ring for each maximal ideal $M$ of $R$.
The examples of locally divided rings abound; zero-dimensional rings,
one-dimensional domains and Pr{\"u}fer domains are locally divided rings.
Moreover, a direct sum of finitely many locally divided rings is
a locally divided ring \cite[Proposition 2.1(b)]{bd01}.

\begin{theorem}\label{one-to-one cor(locally divided)}
Let $R$ be a commutative ring with identity,
$S$ a multiplicative subset of $R$ and
$I$ an ideal of $R$ disjoint from $S$.
If $IR_{S}$ is an $n$-absorbing ideal of $R_{S}$
and ${\rm Sat}_S(I)=I:s$ for some $s\in S$,
then $I$ is an $S$-$n$-absorbing ideal of $R$.
Moreover, the converse holds if $R$ is locally divided.
\end{theorem}

\begin{proof}
Suppose that $IR_{S}$ is an $n$-absorbing ideal of $R_{S}$
and ${\rm Sat}_S(I)=I:s$ for some $s\in S$.
Then $I:s$ is an $n$-absorbing ideal of $R$ \cite[Theorem 4.2(a)]{Anderson}.
Hence $I$ is an $S$-$n$-absorbing ideal of $R$ by Proposition \ref{colon}.
For the remainder argument,
suppose that $R$ is a locally divided ring and $I$ is an $S$-$n$-absorbing ideal of $R$.
Then $IR_S$ is an $n$-absorbing ideal of $R_S$ by Proposition \ref{quotient extension}, and there exists an element $s\in S$ such that
$I:s$ is an $n$-absorbing ideal of $R$ by Proposition \ref{colon}.
Therefore $I:s=Q_{1}\cap\cdots \cap Q_{m}$ for some primary ideals $Q_{1},\dots, Q_{m}$ of $R$ with $m\le n$ \cite[Corollary 13]{c211}.
Thus by Lemma \ref{v31}, ${\rm Sat}_S(I)=I:s$ for some $s\in S$.
\end{proof}

We were unable to decide whether the converse of Theorem \ref{one-to-one cor(locally divided)} holds when the condition `locally divided' is dropped. Hence we raise the following question:

\begin{question}
Is the converse of Theorem \ref{one-to-one cor(locally divided)}
true without $R$ being locally divided?
\end{question}

At the end of this section, we can derive the following.

\begin{corollary}
Let $R$ be a commutative ring with identity,
$S$ a multiplicative subset of $R$ and
$I$ an ideal of $R$ disjoint from $S$.
Then the following assertions hold.
\begin{enumerate}
\item[(1)] If $R$ is an $S$-Laskerian ring, then the following conditions are equivalent.
\begin{enumerate}
\item[(a)] $I$ is an $S$-$n$-absorbing ideal of $R$.
\item[(b)] $IR_S$ is an $n$-absorbing ideal of $R$.
\item[(c)] $I:s$ is an $n$-absorbing ideal of $R$ for some $s \in S$.
\end{enumerate}
\item[(2)] If $R$ is a locally divided ring, then the following conditions are equivalent.
\begin{enumerate}
\item[(a)] $I$ is an $S$-$n$-absorbing ideal of $R$.
\item[(b)] $IR_S$ is an $n$-absorbing ideal of $R$ and ${\rm Sat}_S(I) = I : s$ for some $s \in S$.
\item[(c)] $I:s$ is an $n$-absorbing ideal of $R$ for some $s \in S$.
\end{enumerate}    

\end{enumerate}
\end{corollary}

\section{The functions $\omega_{R,S}(I)$ and $\Omega_S(R)$}\label{sec 4}

Let $R$ be a commutative ring with identity,
$S$ a multiplicative subset of $R$ and
$I$ an ideal of $R$ disjoint from $S$.
It is easy to show that if $I$ is an $S$-$n$-absorbing ideal of $R$,
then $I$ is an $S$-$m$-absorbing ideal of $R$ for all $m \geq n$.
Hence if 
$I$ is an $S$-$n$-absorbing ideal for some $n \in \mathbb{N}$,
then there exists the minimal positive integer $m$ such that
$I$ is an $S$-$m$-absorbing ideal.
Such minimal integer $m$ is denoted by $\omega_{R,S}(I)$ and
we set $\omega_{R,S}(I) = \infty$ if $I$ is not an $S$-$n$-absorbing ideal of $R$
for any $n \in \mathbb{N}$.
Recall that $\omega_R(I)$ indicates the smallest $n$ such that $I$ is an $n$-absorbing ideal of $R$ \cite[p.1649]{Anderson}.
Similarly, $\omega_{R}(I) = \infty$ if $I$ is not an $n$-absorbing ideal of $R$
for any $n \in \mathbb{N}$.
In other words, when $S$ consists of units, then we write $\omega_{R}(I)=\omega_{R,S}(I)$.

At the beginning of this section we examine
simple results that allow us to determine
the value of $\omega_{R,S}(I)$.

\begin{proposition}
Let $R$ be a commutative ring with identity,
$S$ a multiplicative subset of $R$ and
$\varphi : R\to T$ a homomorphsim.
Suppose that $I$ is an ideal of $R$ disjoint from $S$ and
$J$ is an ideal of $\varphi(R)$ disjoint from $\varphi(S)$.
Then the following assertions hold.
\begin{enumerate}
\item[(1)]
If $\ker(\varphi) \subseteq I$,
then $\omega_{R,S}(I) = \omega_{\varphi(R),\varphi(S)}(\varphi(I))$.
\item[(2)]
$\omega_{R,S}(\varphi^{-1}(J)) = \omega_{\varphi(R),\varphi(S)}(J)$.
\item[(3)]
There exists an element $s \in S$ such that
$\omega_{R,S}(I) = \omega_{R}(I:s)$.
\item[(4)]
If $R$ is an $S$-Laskerian ring,
then $\omega_{R,S}(I) = \omega_{R_S}(IR_S)$.
\item[(5)]
If $R$ is a locally divided ring,
then $\omega_{R,S}(I) = \omega_{R}({\rm Sat}_S(I)) = \omega_{R_S}(IR_S)$.
\end{enumerate}
\end{proposition}

\begin{proof}
These results follow directly from Proposition \ref{homo cor} and \ref{colon},
and Theorem \ref{one-to-one cor(Laskerian)} and \ref{one-to-one cor(locally divided)}.
\end{proof}

\begin{proposition}\label{product omega}
Let $R_1, R_2$ be commutative rings with identity and
let $S_1, S_2$ be multiplicative subsets of $R_1,R_2$, respectively.
For $j=1,2$, let $I_j$ be an ideal of $R_j$ disjoint from $S_j$.
Let $R = R_1 \times R_2$, $S = S_1 \times S_2$ and $I = I_1 \times I_2$.
Then $\omega_{R, S}(I)=\omega_{R_1,S_1}(I_1) + \omega_{R_2,S_2}(I_2)$.
\end{proposition}

\begin{proof}
Suppose that $\omega_{R_1,S_1}(I_1) = m$ and
$\omega_{R_2,S_2}(I_2) = n$.
We first consider the case when $m,n\in\mathbb{N}$.
Then there exist elements
$a_1,\dots,a_m \in R_1$ and $b_1,\dots,b_n \in R_2$ such that
$a_1\cdots a_m \in I_1$, $b_1 \cdots b_n \in I_2$ and
$s\prod_{\substack{1 \leq \ell \leq m \\ \ell \neq i}}a_{\ell} \notin I_1$ and
$t\prod_{\substack{1 \leq \ell \leq m \\ \ell \neq j}}b_{\ell} \notin I_2$ for all 
$1 \leq i \leq m$, $1 \leq j \leq n$, $s \in S_1$ and $t \in S_2$.
Hence
$\big(\prod_{1 \leq \ell \leq m}(a_{\ell},1)\big)
\big(\prod_{1\leq \ell \leq n}(1,b_{\ell})\big) =
(a_1\cdots a_m, b_1\cdots b_n) \in I_1 \times I_2$,
but neither
$(s,t)\big(\prod_{\substack{1 \leq \ell \leq m \\ \ell \neq i}}(a_{\ell},1)\big)
\big(\prod_{1\leq \ell \leq n}(1,b_{\ell})\big)$ nor
$(s,t)\big(\prod_{1\leq \ell \leq m}(a_{\ell},1)\big)
\big(\prod_{\substack{1 \leq \ell \leq m \\ \ell \neq j}}(1,b_{\ell})\big)$
belong to $I_1 \times I_2$  for all 
$1 \leq i \leq m$, $1 \leq j \leq n$.
Therefore $m+n \leq \omega_{R,S}(I)$.
For the reverse inequality,
let $N = m+n+1$ and
suppose that $(a_1,b_1),\dots,(a_N,b_N) \in R_1 \times R_2$ with
$\prod_{1 \leq \ell \leq N}(a_{\ell},b_{\ell}) \in I_1 \times I_2$.
Then there are $\{i_1,\dots,i_m\}, \{j_1,\dots,j_n\} \subseteq \{1,\dots,N\}$
such that $s_1 a_{i_1} \cdots a_{i_m} \in I_1$ and
$s_2 b_{j_1} \cdots b_{j_n} \in I_2$ for some $s_1,s_2 \in S$.
Let $K = \{i_1, \dots, i_m\} \cup \{j_1, \dots, j_n\}$.
Then $(s_1,s_2) \prod_{i \in K} (a_i,b_i) \in I_1 \times I_2$.
Since $|K| \leq m+n$,
$\omega_{R,S}(I) \leq m+n$.
Thus $\omega_{R,S}(I) = m+n = \omega_{R_1,S_1}(I_1) + \omega_{R_2,S_2}(I_2)$.
This proof can be easily adapted to show that $\omega_{R,S}(I) =\infty$
if and only if $m=\infty$ or $n=\infty$, which yields the desired conclusion.
\end{proof}

\begin{corollary}\label{product S-n-absorbing}
Let $R_1, R_2$ be commutative rings with identity and
let $S_1, S_2$ be multiplicative subsets of $R_1,R_2$, respectively.
For $j=1,2$, let $I_j$ be an ideal of $R_j$ disjoint from $S_j$.
If $I_1$ is an $S_1$-$m$-absorbing ideal of $R_1$ and
$I_2$ is an $S_2$-$n$-absorbing ideal of $R_2$,
then $I_1 \times I_2$ is an $(S_1 \times S_2)$-$(m+n)$-absorbing ideal of $R_1 \times R_2$.
\end{corollary}

\begin{proof}
This result follows directly from Proposition \ref{product omega}.
\end{proof}

\begin{theorem}
Let $R$ be a commutative ring with identity,
$S$ a multiplicative subset of $R$ and
$I$ an ideal of $R$ whose minimal prime ideals are disjoint from $S$.
Suppose that $I$ is an $S$-$n$-absorbing ideal of $R$ that has exactly $n$ minimal prime ideals
$P_1,\dots,P_n$.
Then there exists an element $s \in S$ such that
$sP_1\cdots P_n \subseteq I$.
Moreover, $\omega_{R,S}(I) = n$.
\end{theorem}

\begin{proof}
As $I$ is an $S$-$n$-absorbing ideal of $R$,
there exists an $s \in S$ such that $I:s$ is an $n$-absorbing ideal of $R$ by Proposition \ref{colon}.
Note that for each $1\leq i\leq n$,
$P_i$ is a minimal prime ideal of $I:s$
since $P_1,\dots,P_n$ are disjoint from $S$,
which means that $\omega_{R}(I:s) \geq n$ \cite[Theorem 2.5]{Anderson}.
Hence $\omega_{R,S}(I) \geq n$,
and it follows that $\omega_{R,S}(I) = n$.
Also, we obtain $\omega_{R}(I:s) = n$.
It follows that there are at most $n$ minimal prime ideals of $I :s$
\cite[Theorem 2.5]{Anderson}.
Therefore $I : s$ has exactly $n$ minimal prime ideals,
so $P_1\cdots P_n\subseteq I:s$ \cite[Theorem 2.14]{Anderson}.
Thus $sP_1\cdots P_n\subseteq I$.
\end{proof}

The next remark shows the existence of ideal $I$ of $R$ with $\omega_{R,S}(I) = \infty$.

\begin{remark}\label{omage infinite}
{\rm 
Let $R$ be a commutative ring with identity and
let $S$ be a multiplicative subset of $R$.
Note that for an ideal $I$ of $R$ disjoint from $S$,
$I$ may not be an $S$-$n$-absorbing ideal of $R$ for any $n\in\mathbb{N}$
while $IR_{S}$ is $n$-absorbing of $R_{S}$ for some $n\in\mathbb{N}$.
Indeed, let $I,S$ and $R$ be as in Example \ref{counter example of converse quotient}.
As mentioned in Example \ref{counter example of converse quotient},
$IR_{S}$ is a prime ideal ({\it i.e.}, a $1$-absorbing ideal) of $R_{S}$,
while $I$ is not an $S$-prime ideal ({\it i.e.},
an $S$-$1$-absorbing ideal) of $R$.
Since $R$ is locally divided, ${\rm Sat}_S(I)\neq I:s$
for any $s\in S$ by Theorem \ref{one-to-one cor(locally divided)}.
Then again by Theorem \ref{one-to-one cor(locally divided)},
$I$ is not an $S$-$n$-absorbing ideal of $R$ for any $n\in\mathbb{N}$.
}
\end{remark}

Let $R$ be a commutative ring with identity and
let $S$ be a multiplicative subset of $R$.
Define $\Omega_S(R) = \{\omega_{R,S}(I) \,|\, I {\rm \ is \ an \ ideal \ of \ } R {\rm \ disjoint \ from \ } S\}$.
As Remark \ref{omage infinite} shows, we have $\Omega_S(R) \not\subseteq \mathbb{N}$ in general.
Naturally, we seek the conditions on a ring $R$ such that 
$\Omega_S(R) \subseteq \mathbb{N}$, which means that
for each ideal $I$ of $R$ disjoint from $S$,
there exists a positive integer $n \in \mathbb{N}$ such that
$I$ is an $S$-$n$-absorbing ideal of $R$.

Recall that $R$ is an {\it $S$-Noetherian ring} if for each ideal $I$ of $R$,
there exist an element $s\in S$ and a finitely generated ideal $J$ of $R$
such that $sI\subseteq J\subseteq I$ \cite{ad02}.
Also, recall that $R$ is said to be {\it arithmetical}
if $(A+B)\cap C=(A\cap C)+(B\cap C)$ 
for all ideals $A,B$ and $C$ of $R$.
Equivalently, $R$ is arithmetical if and only if
$R_{M}$ is a chained ring for all maximal ideals $M$ of $R$ \cite[Theorem 1]{j66}.
Consequently, every arithmetical ring is locally divided,
and an integral domain is an arithmetical ring if and only if it is a Pr{\"u}fer domain.

\begin{theorem} \label{sum}
Let $R$ be a commutative ring with identity and
let $S$ be a multiplicative subset of $R$.
Then the following assertions hold.
\begin{enumerate}
\item[(1)] If $R$ is a strongly $S$-Laskerian ring,
then $\Omega_S(R) \subseteq \mathbb{N}$.
Moreover, the converse holds if $R$ is locally divided.
\item[(2)] If $R$ is an arithmetical ring, then the following are equivalent.
\begin{enumerate}
\item[(a)] $R$ is an $S$-Noetherian ring.
\item[(b)] $R$ is a strongly $S$-Laskerian ring.
\item[(c)] $\Omega_S(R) \subseteq \mathbb{N}$.
\end{enumerate}
\end{enumerate}
\end{theorem}

\begin{proof}
(1) Suppose that $R$ is a strongly $S$-Laskerian ring
and let $I$ be an ideal of $R$ disjoint from $S$.
Then $R_{S}$ is a strongly Laskerian ring and
there exists an element $s\in S$ such that ${\rm Sat}_S(I)=I:s$ \cite[Proposition 3.2]{v22}.
Hence $IR_{S}$ is an $n$-absorbing ideal of $R_{S}$
for some $n\in\mathbb{N}$ \cite[Lemma 19]{c211}.
Therefore by Theorem \ref{one-to-one cor(locally divided)},
$I$ is an $S$-$n$-absorbing ideal of $R$.
Hence $\Omega_S(R) \subseteq \mathbb{N}$.
For the remaining argument,
let $R$ be a locally divided ring and
suppose that $\Omega_S(R) \subseteq \mathbb{N}$.
Then $R_{S}$ is a locally divided ring \cite[Proposition 2.1(a)]{bd01}
and for each ideal $I$ of $R$ disjoint from $S$,
$IR_{S}$ is an $n$-absorbing ideal of $R_{S}$ and
${\rm Sat}_S(I)=I:s$ for some $s\in S$ by Theorem \ref{one-to-one cor(locally divided)}.
Hence $R_{S}$ is strongly Laskerian \cite[Lemma 20(2)]{c211},
and thus $R$ is a strongly $S$-Laskerian ring
\cite[Proposition 3.2]{v22}.

(2) Let $R$ be an arithmetical ring.
Note that the implications (a)$\Rightarrow$(b)$\Rightarrow$(c)
follow from (1) and \cite[Corollary 3.3]{v22}.
Suppose that $\Omega_S(R) \subseteq \mathbb{N}$.
Then by Theorem \ref{one-to-one cor(locally divided)},
each ideal of $R_{S}$ is $n$-absorbing for some $n\in\mathbb{N}$,
so $R_{S}$ is Noetherian \cite[Corollary 36]{c211}.
Moreover, given an ideal $I$ of $R$ disjoint form $S$,
we have ${\rm Sat}_S(I)=I:s$ for some $s\in S$ by Theorem \ref{one-to-one cor(locally divided)}.
Thus $R$ is an $S$-Noetherian ring
\cite[Proposition 2(f)]{ad02}.
\end{proof}

\begin{remark}
{\rm
A locally divided strongly $S$-Laskerian ring may not be $S$-Noetherian.
Indeed, let $R=\mathbb{Q}+X\mathbb{R}[X]$, and $S=R\setminus X\mathbb{R}[X]$.
Since $R$ is strongly Laskerian \cite[Corollary 7]{bf82}, it is strongly $S$-Laskerian.
On the other hand, $R_{S}$ is not a Noetherian ring,
so $R$ is not $S$-Noetherian \cite[Proposition 2(f)]{ad02}.
}
\end{remark}

\begin{lemma}\label{saturation}
Let $R$ be a commutative ring with identity,
$S$ a multiplicative subset of $R$ with saturation $\overline{S}$ and
$I$ an ideal of $R$ disjoint from $S$.
Then $I$ is an $S$-$n$-absorbing ideal of $R$ if and only if
$I$ is an $\overline{S}$-$n$-absorbing ideal of $R$.
\end{lemma}

\begin{proof}
The `only if' part is obvious.
Suppose that $I$ is an $\overline{S}$-$n$-absorbing ideal of $R$.
Let $r_1,\dots,r_{n+1} \in R$ with $r_1 \cdots r_{n+1} \in I$.
Then we may assume that
there exists an element $s \in \overline{S}$ such that
$sr_1 \cdots r_n \in I$.
Since $\overline{S}$ is the saturation of $S$,
there exist an element $t \in \overline{S}$ such that
$st \in S$.
Thus $I$ is an $S$-$n$-absorbing ideal of $R$ since
$str_1 \cdots r_n \in I$ and $st \in S$.
\end{proof}

\begin{theorem}\label{what}
Let $R$ be a commutative ring with identity which is a chained ring and
let $S$ be a multiplicative subset of $R$ with saturation $\overline{S}$.
If $\Omega_S(R) \subseteq \mathbb{N}$,
then either every nonunit regular element of $R$ belongs to $\overline{S}$
or $R$ is a Noetherian ring.
\end{theorem}

\begin{proof}
Suppose that there exists a nonunit regular element $r$ of $R$ which does not belong to $\overline{S}$.
Then we can choose a prime ideal $P$ of $R$ disjoint from $\overline{S}$ which contains $r$;
that is, $P$ is a regular prime ideal of $R$ disjoint from $\overline{S}$.
Note that by Proposition \ref{sum}(2) and Lemma \ref{saturation},
$R$ is $\overline{S}$-Noetherian.
Hence there exist an element $s\in \overline{S}$
and a finitely generated ideal $J$ of $R$
such that $sP\subseteq J\subseteq P$.
Since $sP=P$ \cite[Exercise 17.3]{G},
$P$ is finitely generated and must be the maximal ideal of $R$ \cite[Exercise 17.3]{G}.
Then every element of $\overline{S}$ is a unit in $R$
since $R$ is a quasi-local ring and $\overline{S}$ is disjoint from $P$.
Thus $R$ is a Noetherian ring.
\end{proof}

\begin{corollary}
Let $R$ be a valuation domain and
let $S$ be a multiplicative subset of $R$ with
saturation $\overline{S}$.
Then the following are equivalent.
\begin{enumerate}
\item[(1)] $\Omega_S(R) \subseteq \mathbb{N}$.
\item[(2)] Either $R_{S}$ is a field or $R$ is a Noetherian ring.
\end{enumerate}
\end{corollary}

\begin{proof}
Suppose that (1) holds.
If every nonzero nonunit belongs to $\overline{S}$,
then $R_{S}$ is a field.
If some nonzero nonunit does not belong to $\overline{S}$,
then $R$ is Noetherian by Proposition \ref{what}.
Conversely, if $R_{S}$ is a field or $R$ is Noetherian,
then $R$ is $S$-Noetherian, and the conclusion follows from Proposition \ref{sum}.
\end{proof}

\section{$S$-$n$-absorbing ideals of amalgamation}\label{sec 5}

Let $A,B$ be commutative rings with identity,
$J$ an ideal of $B$ and
$f : A \to B$ a homomorphism.
Then we can define the following subring of $A \times B$:
\begin{center}
$A \bowtie^f J = \{(a,f(a)+j) \,|\, a \in A, j \in J\}$.
\end{center}
$A \bowtie^f J$ is called the {\it amalgamation of $A$ with $B$ along $J$ with respect to $f$}.
The readers can refer to \cite{Anna, Anna1, Issoual}
for amalgamated algebras.
This section begins by examining the conditions under which
the well-known ideals of $A \bowtie^f J$
become $S$-$n$-absorbing ideals.

Let $I$ be an ideal of $A$,
$S$ a multiplicative subset of $A$ and
$K$ an ideal of $f(A)+J$.
Throughout this section,
we set
\begin{center}
$S^{{\bowtie}^f} := \{(s,f(s)) \,|\, s \in S\}$,\\
$I \bowtie^f J := \{(i,f(i) + j) \,|\, i \in I, j \in J\}$,\\
$\overline{K}^f := \{(a,f(a)+j) \,|\, a \in A,j \in J, f(a) +j \in K\}$
\end{center}
and
\begin{center}
$\overline{I \times K}^f := \{(i,f(i) +j) \,|\, i \in I , j \in J, f(i) + j \in K\}$.
\end{center}
Then clearly, $S^{{\bowtie}^f}$ is a multiplicative subset of $A \bowtie^f J$; and
$I \bowtie^f J, \overline{K}^f$ and $\overline{I \times K}^f$
are ideals of $A \bowtie^f J$.

\begin{lemma}
Let $A, B$ be commutative rings with identity,
$S$ a multiplicative subset of $A$ and
$f : A \to B$ a homomorphism.
Let $I$ be an ideal of $A$,
$J$ an ideal of $B$ and
$K$ an ideal of $f(A) + J$.
Then the following assertions hold.
\begin{enumerate}
\item[(1)]
$I \cap S = \emptyset$ if and only if
$(I \bowtie^f J) \cap S^{{\bowtie}^f} = \emptyset$.
\item[(2)]
$K \cap f(S) = \emptyset$ if and only if
$\overline{K}^f \cap S^{{\bowtie}^f} = \emptyset$.
\item[(3)]
If $I \cap S = \emptyset$ and $K \cap f(S) = \emptyset$, then
$\overline{I \times K}^f \cap S^{{\bowtie}^f} = \emptyset$.
\end{enumerate}
\end{lemma}

\begin{proof}
The results are straightforward.
\end{proof}

\begin{theorem}\label{basic amalgamation}
Let $A, B$ be commutative rings with identity,
$S$ a multiplicative subset of $A$ and
$f : A \to B$ a homomorphism.
Let $I$ be an ideal of $A$ disjoint from $S$,
$J$ an ideal of $B$ and
$K$ an ideal of $f(A) + J$ disjoint from $f(S)$.
Then the following assertions hold.
\begin{enumerate}
\item[(1)]
$I$ is an $S$-$n$-absorbing ideal  if and only if
$I \bowtie^f J$ is an $S^{{\bowtie}^f}$-$n$-absorbing ideal of $A \bowtie^f J$.
\item[(2)]
$K$ is an $f(S)$-$n$-absorbing ideal of $f(A) + J$ if and only if
$\overline{K}^f$ is an $S^{{\bowtie}^f}$-$n$-absorbing ideal of $A \bowtie^f J$.
\item[(3)]
If $I$ is an $S$-$m$-absorbing ideal of $A$ and
$K$ is an $f(S)$-$n$-absorbing ideal of $f(A)+J$,
then $\overline{I \times K}^f$ is an
$S^{{\bowtie}^f}$-$(m+n)$-absorbing ideal of $A \bowtie^f J$.
\end{enumerate}
\end{theorem}

\begin{proof}
(1) Define $\varphi : A \to (A \bowtie^f J)/(\{0\} \bowtie^f J)$ by
$\varphi(a) = (a,f(a)) + (\{0\} \bowtie^f J)$.
Then $\varphi$ is an isomorphism and
$\varphi(S) = S^{{\bowtie}^f}/(\{0\} \bowtie^f J)$.
Hence by Proposition \ref{homo cor}(1),
$I$ is an $S$-$n$-absorbing ideal of $A$ if and only if
$(I \bowtie^f J)/(\{0\} \bowtie^f J)$ is a $\varphi(S)$-$n$-absorbing ideal of $(A \bowtie^f J)/(\{0\} \bowtie^f J)$.
Also, by Corollary \ref{factor cor},
$(I \bowtie^f J)/(\{0\} \bowtie^f J)$ is an $\varphi(S)$-$n$-absorbing ideal of $(A \bowtie^f J)/(\{0\} \bowtie^f J)$
if and only if
$I \bowtie^f J$ is an $S^{{\bowtie}^f}$-$n$-absorbing ideal of $A \bowtie^f J$.
Thus the result holds.

(2) Define $\varphi : A \bowtie^f J \to f(A) + J$ by
$\varphi(a,f(a) + j) = f(a) + j$.
Then $\varphi$ is an epimorphism with $\ker(\varphi) = f^{-1}(J) \times \{0\}$,
$\varphi(\overline{K}^f) = K$ and $\varphi(S^{{\bowtie}^f}) = f(S)$.
Since $f^{-1}(J) \times \{0\} \subseteq \overline{K}^f$,
we obtain that
$K$ is an $f(S)$-$n$-absorbing ideal of $f(A) + J$ if and only if
$\overline{K}^f$ is an $S^{{\bowtie}^f}$-$n$-absorbing ideal of $A \bowtie^f J$
by Proposition \ref{homo cor}(1).

(3) By Corollary \ref{product S-n-absorbing},
$I \times K$ is an $(S \times f(S))$-$(m+n)$-absorbing ideal of $A \times (f(A) + J)$.
Let $a_1,\dots,a_{m+n+1} \in A \times (f(A) + J)$ with
$a_1 \cdots a_{m+n+1} \in I \times K$.
Then we may assume that
there exist elements $s_1,s_2 \in S$ such that
$(s_1,f(s_2))a_1 \cdots a_{m+n} \in I \times K$.
Hence $(s_1s_2, f(s_1s_2)) a_1 \cdots a_{m+n} \in I \times K$.
Therefore $I \times K$ is an $S^{{\bowtie}^f}$-$(m+n)$-absorbing ideal of $A \times (f(A) + J)$.
Since $A \times (f(A) + J)$ is a ring extension of $A \bowtie^f J$ and
$\overline{I \times K}^f= (I\times K) \cap (A \bowtie^f J)$,
$\overline{I \times K}^f$ is an $S^{{\bowtie}^f}$-$(m+n)$-absorbing ideal of $A \bowtie^f J$
by Proposition \ref{basic property}(2).
\end{proof}

By Theorem \ref{basic amalgamation},
we obtain following example and some results.

\begin{example}
{\rm
Let $A,B$ be commutative rings with identity such that $A \subseteq B$ and
let $S$ be a multiplicative subset of $A$.
Set $J_1 = XB[X]$ and $J_2 = XB[\![X]\!]$.

(1) Consider the natural embedding $\iota_1 : A \to B[X]$.
Then it is easy to check that $A \bowtie^{\iota_1} J_1$ is isomorphic to $A + XB[X]$.
Hence by Theorem \ref{basic amalgamation}(1),
$I$ is an $S$-$n$-absorbing ideal of $A$ if and only if
$I+XB[X]$ is an $S$-$n$-absorbing ideal of $A+XB[X]$.

(2) Consider the natural embedding $\iota_2 : A \to B[\![X]\!]$.
Then it is easy to check that $A \bowtie^{\iota_2} J_2$ is isomorphic to $A + XB[\![X]\!]$.
Hence by Theorem \ref{basic amalgamation}(1),
$I$ is an $S$-$n$-absorbing ideal of $A$ if and only if
$I+XB[\![X]\!]$ is an $S$-$n$-absorbing ideal of $A+XB[\![X]\!]$.
}
\end{example}

\begin{corollary}
Let $A, B$ be commutative rings with identity,
$S$ a multiplicative subset of $A$ and
$f : A \to B$ a homomorphism.
Let $I$ be an ideal of $A$ disjoint from $S$,
$J$ an ideal of $B$ and
$K$ an ideal of $f(A) + J$ disjoint from $f(S)$.
Then the following assertions hold.
\begin{enumerate}
\item[(1)]
Every $S^{{\bowtie}^f}$-$n$-absorbing ideal of $A \bowtie^f J$ containing $\{0\} \times J$
is of the form
$I \bowtie^f J$, where $I$ is an $S$-$n$-absorbing ideal of $A$.
\item[(2)]
Every $S^{{\bowtie}^f}$-$n$-absorbing ideal of $A \bowtie^f J$ containing $f^{-1}(J) \times \{0\}$
is of the form
$\overline{K}^f$, where $K$ is an $f(S)$-$n$-absorbing ideal of $f(A) + J$.
\end{enumerate}
\end{corollary}

\begin{proof}
Let $L$ be an ideal of $A \bowtie^f J$ disjoint from $S^{{\bowtie}^f}$.
Define $\pi_1 : A \bowtie^f J \to A$ by $\pi_1(a,f(a)+j) = a$ and
define $\pi_2 : A \bowtie^f J \to f(A) + J$ by $\pi_2(a,f(a)+j) = f(a)+j$.
Then $\pi_1$ and $\pi_2$ are epimorphisms.

(1) By Theorem \ref{basic amalgamation}(1),
$I \bowtie^f J$ is an $S^{{\bowtie}^f}$-$n$-absorbing ideal of $A \bowtie^f J$
for any $S$-$n$-absorbing ideal $I$ of $A$.
Suppose that $L$ is an $S^{{\bowtie}^f}$-$n$-absorbing ideal of $A \bowtie^f J$ containing $\{0\} \times J$.
It is easy to show that $L = \pi_1(L) \bowtie^f J$.
Since $\ker(\pi_1) = \{0\} \times J$, we obtain that
$\pi_1(L)$ is an $S$-$n$-absorbing ideal of $A$ by Proposition \ref{homo cor}(1).
Thus the result holds.

(2) By Theorem \ref{basic amalgamation}(2),
$\overline{K}^f$ is an $S^{\bowtie^f}$-$n$-absorbing ideal of $A \bowtie^f J$
for any $f(S)$-$n$-absorbing ideal $K$ of $f(A) +J$.
Suppose that $L$ is an $S^{{\bowtie}^f}$-$n$-absorbing ideal of $A \bowtie^f J$ containing $f^{-1}(J) \times \{0\}$.
It is easy to show that $L = \overline{\pi_2(L)}^f$.
Since $\ker(\pi_2) = f^{-1}(J) \times \{0\}$, we obtain that
$\pi_2(L)$ is an $f(S)$-$n$-absorbing ideal of $f(A) + J$ by Proposition \ref{homo cor}(1).
Thus the result holds.
\end{proof}

Let $D$ and $T$ be integral domains with $D \subseteq T$,
let $\{M_{\alpha} \,|\, \alpha \in \Lambda\}$ a subset of ${\rm Max}(T)$ and
$J = \bigcap_{\alpha \in \Lambda}M_{\alpha}$ with $J \cap D = (0)$.
Consider the natural embedding $\iota : D \to T$.
Then it is easy to show that $D+J$ is canonically isomorphic to $D \bowtie^{\iota}J$.
Under these notations,
we obtain the following result directly from Theorem \ref{basic amalgamation}(1).

\begin{corollary}
Let $S$ be a multiplicative subset of $D$,
$I$ an ideal of $D$ disjoint from $S$ and
$J$ as above.
Then $I$ is an $S$-$n$-absorbing ideal of $D$ if and only if
$I+J$ is an $S^{\bowtie^{\iota}}$-$n$-absorbing ideal of $D+J$.
\end{corollary}

Let $A,B$ be commutative rings with identity,
$J$ an ideal of $B$ and
$f : A \to B$ a homomorphism.
Consider the ideals $I$ and $H$ of $A$ and $f(A) + J$, respectively,
such that $f(I)J \subseteq H \subseteq J$.
Obviously, $I \bowtie^f H : = \{(i,f(i) + h \,|\, i \in I {\rm \ and \ } h \in H \}$
is an ideal of $A \bowtie^f J$.
Note that $I \bowtie^f H$ is generally not an $n$-absorbing ideal of $A \bowtie^f J$
when $I$ is an $n$-absorbing ideal of $A$ \cite[Example 2.7]{Issoual}.
Similarly, we attach the example that shows that
$I \bowtie^f H$ is generally not an $S^{\bowtie^f}$-$n$-absorbing ideal of $A \bowtie^f J$
when $I$ is an $S$-$n$-absorbing ideal of $A$.
We mimic that of \cite[Example 2.7]{Issoual}, but the proof is included for the sake of completeness.

\begin{example}{\rm (\cite[Example 2.7]{Issoual})}
{\rm
Let $A = \mathbb{Z}$, $B = \mathbb{Q}[\![X]\!]$ and
$f : A \to B$ be the natural embedding.
Set $J = X\mathbb{Q}[\![X]\!]$, $I = (0)$ and
$H = \{Xg \,|\, g \in f(A) + J \}$.
Let $S$ be a multiplicative subset of $A$ disjoint from $I$ with $\overline{S} \neq \mathbb{Z} \setminus \{0\}$.
Since $I$ is a prime ideal of $A$,
it is an $S$-$n$-absorbing ideal of $A$.
We claim that $I \bowtie^f H$ is not an $S^{\bowtie^f}$-$n$-absorbing ideal of $A \bowtie^f J$.
Let $a \in \mathbb{Z} \setminus \overline{S} \cup \{0\}$.
Then $(a,a)^n(0, \frac{X}{a^n}) = (0,X) \in I \bowtie^f H$.
It is clear that $(s,f(s))(a,a)^n \notin I \bowtie^f H$ for all $s \in S$.
Now, suppose that there exists $s \in S$ such that
$(0, \frac{s}{a}X) = (s,f(s))(a,a)^{n-1}(0, \frac{X}{a^n}) \in I \bowtie^f H$.
This implies that $\frac{s}{a} \in \mathbb{Z}$.
Hence there exists $b \in \mathbb{Z}$ such that $ab = s \in S \subseteq \overline{S}$,
which means that $a,b \in \overline{S}$.
This contradicts our choice of $a$.
Therefore $(s,f(s))(a,a)^{n-1}(0, \frac{X}{a^n}) \notin I \bowtie^f H$ for all $s \in S$.
Thus $I \bowtie^f H$ is not an $S^{\bowtie^f}$-$n$-absorbing ideal of $A \bowtie^f J$.
}
\end{example}

Now, we investigate the conditions for $I \bowtie^f H$
to be an $S^{\bowtie^f}$-$n$-absorbing ideal of $A \bowtie^f J$.
To do this, we need a couple of lemmas.

\begin{lemma}\label{P-primary lemma}
Let $R$ be a commutative ring $($not necessary with identity$)$,
$I$ an ideal of $R$ and
$S$ a multiplicative subset of $R$ disjoint from $I$.
If $P$ is a prime ideal of $R$ and
$I$ is a $P$-primary ideal of $R$ such that $sP^n \subseteq I$ for some $s \in S$,
then $I$ is an $S$-$n$-absorbing ideal of $R$.
\end{lemma}

\begin{proof}
Let $a_1,\dots, a_{n+1} \in R$ with $a_1 \cdots a_{n+1} \in I$.
Suppose that there exists $1 \leq i \leq n$ such that $a_i \notin P$.
Since $I$ is $P$-primary, $\prod_{\substack{1 \leq j \leq n\\ j \neq i}}a_j \in I$,
which means that $s\prod_{\substack{1 \leq j \leq n\\ j \neq i}}a_j \in I$ for all $s \in S$.
For the remainder case, suppose that $a_1, \dots, a_{n+1} \in P$.
Then $sa_1 \cdots a_{n+1} \in sP^n \subseteq I$.
Thus $I$ is an $S$-$n$-absorbing ideal of $R$.
\end{proof}

The following lemma is the generalization of Choi and Walker \cite[Theorem 10]{choi},
which proves one of the three conjectures raised by Anderson and Badawi \cite{Anderson}.

\begin{lemma}\label{conjecture 3}
Let $R$ be a commutative ring $($not necessary with identity$)$,
$S$ a multiplicative subset of $R$ and
$I$ an ideal of $R$.
If $I$ is an $S$-$n$-absorbing ideal of $R$ and
associated to $s$,
then $s(\sqrt{I})^n \subseteq I$.
\end{lemma}

\begin{proof}
Suppose that $I$ is an $S$-$n$-absorbing ideal of $R$ and associated to $s$.
Then $I:s$ is an $n$-absorbing ideal of $R$ by Proposition \ref{colon}.
Hence $(\sqrt{I:s})^n\subseteq I:s$ \cite[Theorem 10]{choi}.
As $\sqrt{I}\subseteq \sqrt{I:s}$,
we have $(\sqrt{I})^n\subseteq I:s$.
Thus $s(\sqrt{I})^n\subseteq I$.
\end{proof}

\begin{lemma}{\rm (\cite[Lemma 2.9]{Issoual})}\label{local ring lemma}
Let $(A,\mathfrak{m})$ be a local ring,
$f : A \to B$ be a ring homomorphism and
$J$ be an ideal of $B$ such that
$f^{-1}(Q) \neq \mathfrak{m}$
for every $Q \in {\rm Spec}(B) \setminus V(J)$.
Consider $I$ $($respectively, $H)$ an ideal of $A$ $($respectively, $f(A) +J)$
such that $f(I)J \subseteq H \subseteq J$.
Then the following statements hold:
\begin{enumerate}
\item[(1)]
If $\sqrt{I} = \mathfrak{m}$,
then $\sqrt{I \bowtie^f H} = \mathfrak{m} \bowtie^f J$.
\item[(2)]
$I$ is $\mathfrak{m}$-primary ideal of $A$ if and only if
$I \bowtie^f H$ is $(\mathfrak{m} \bowtie^f J)$-primary ideal of $A \bowtie^f J$.
\end{enumerate}
\end{lemma}

We conclude this section with the following results.

\begin{theorem}\label{local amalgam}
Let $A,B$ be commutative rings with identity and
let $f :A \to B$ be a homomorpism.
Let $(A,\mathfrak{m})$ be a quasi-local ring,
$S$ a multiplicative subset of $A$ and
$I$ an $\mathfrak{m}$-primary ideal of $A$.
Suppose that $J$ is an ideal of $B$ such that
$f^{-1}(Q) \neq \mathfrak{m}$ for every $Q \in {\rm Spec}(B)\setminus V(J)$ and
$H$ is an ideal of $f(A) + J$ such that $0 = f(\mathfrak{m})J \subseteq H \subseteq J$.
If there exists $s \in S$ such that $f(s) J^n \subseteq H$,
then the following assertions are equivalent.
\begin{enumerate}
\item[(1)]
$I$ is an $S$-$n$-absorbing ideal of $A$ and
associated to $s$.
\item[(2)]
$I \bowtie^f H$ is an $S^{\bowtie^f}$-$n$-absorbing ideal of $A \bowtie^f J$ and
associated to $(s,f(s))$.
\end{enumerate}
\end{theorem}

\begin{proof}
Suppose that $I$ is an $S$-$n$-absorbing ideal of $A$ and associated to $s$ for some $s \in S$.
By Lemma \ref{conjecture 3}, we obtain
$s (\sqrt{I})^n \subseteq I$.
Since $I$ is an $\mathfrak{m}$-primary ideal,
we obtain $s\mathfrak{m}^n \subseteq I$.
Hence by Lemma \ref{local ring lemma}(1),
$(s,f(s))(\sqrt{I\bowtie^f H})^n = (s,f(s))(\mathfrak{m}\bowtie^f J)^n
= (s,f(s))(\mathfrak{m}^n \bowtie^f J^n) \subseteq (I \bowtie^f H)$.
Also, by Lemma \ref{local ring lemma}(2),
$I \bowtie^f H$ is an $(\mathfrak{m} \bowtie^f J)$-primary ideal of $A \bowtie^f J$.
Hence by Lemma \ref{P-primary lemma},
$I \bowtie^f H$ is an $S^{\bowtie^f}$-$n$-absorbing ideal of $A \bowtie^f J$ and associated to $(s, f(s))$.
For the converse, suppose that
$I \bowtie^f H$ is an $S^{\bowtie^f}$-$n$-absorbing ideal of $A \bowtie^f J$ and associated to $(s, f(s))$.
By Lemma \ref{conjecture 3},
$(s,f(s))(\sqrt{I \bowtie^f H})^n \subseteq I \bowtie^f H$.
Since $I$ is an $\mathfrak{m}$-primary ideal of $R$,
$(s,f(s))(\mathfrak{m} \bowtie^f J)^n \subseteq I \bowtie^f H$ by Lemma \ref{local ring lemma}(1).
Hence $s\mathfrak{m}^n \subseteq I$,
and thus $I$ is an $S$-$n$-absorbing ideal of $A$ and associated to $s$
by Lemma \ref{P-primary lemma}.
\end{proof}

\acknowledgments
H. S. Choi was supported by the National Research Foundation of Korea (NRF) grant funded by the
Korea government (MSIT)(No. 2022R1C1C2009021).

\end{document}